\documentclass[12pt,a4paper]{amsart}

\usepackage[utf8]{inputenc}
\usepackage{enumitem,a4,titletoc}
\usepackage{amssymb,bbm,graphicx,mathrsfs}
\usepackage[usenames,dvipsnames]{xcolor}
\usepackage[titletoc]{appendix}
\usepackage[pagebackref,colorlinks,linkcolor=BrickRed,citecolor=OliveGreen,urlcolor=blue,hypertexnames=true]{hyperref}

\DeclareMathOperator{\id}{id}

\DeclareMathOperator{\End}{End}
\DeclareMathOperator{\im}{im}

\DeclareMathOperator{\GL}{GL}
\DeclareMathOperator{\Ortho}{O}
\DeclareMathOperator{\Unit}{U}

\DeclareMathOperator{\dom}{dom}
\DeclareMathOperator{\rad}{rad}
\DeclareMathOperator{\tr}{tr}
\newcommand{\one}{{\bf 1}}
\newcommand{\N}{\mathbb{N}}
\newcommand{\R}{\mathbb{R}}
\newcommand{\C}{\mathbb{C}}

\newcommand{\cA}{\mathcal{A}}
\newcommand{\cB}{\mathcal{B}}
\newcommand{\cC}{\mathcal{C}}
\newcommand{\cD}{\mathcal{D}}

\newcommand{\cN}{\mathcal{N}}
\newcommand{\cR}{\mathcal{R}}

\newcommand{\cO}{\mathcal{O}}
\newcommand{\cS}{\mathcal{S}}

\newcommand{\ee}{\mathbf{e}}
\newcommand{\bb}{\mathbf{b}}
\newcommand{\cc}{\mathbf{c}}

\newcommand{\rr}{\mathbbm r}
\newcommand{\rs}{\mathbbm s}
\newcommand{\rf}{\mathbbm f}
\newcommand{\kk}{\mathbbm k}
\newcommand{\kkb}{\overline{\mathbbm k}}
\newcommand{\Dom}{\mathfrak{Dom}_0}
\newcommand{\ulx}{\boldsymbol{x}}
\newcommand{\ulxx}{\boldsymbol{\xi}}
\newcommand{\ulu}{\boldsymbol{u}}
\newcommand{\ulv}{\boldsymbol{v}}
\newcommand{\uly}{\boldsymbol{y}}

\newcommand{\fl}{\mathscr{Z}}
\newcommand{\fls}{\mathscr{Z}^{\operatorname{s}}}
\newcommand{\flh}{\mathscr{Z}^{\operatorname{h}}}

\newcommand{\Langle}{\mathop{<}\!}
\newcommand{\Rangle}{\!\mathop{>}}
\newcommand{\mx}{\Langle \ulx\Rangle}
\newcommand{\px}{\kk\!\Langle \ulx\Rangle}
\newcommand{\pxc}{\C\!\Langle \ulx\Rangle}
\newcommand{\pxr}{\R\!\Langle \ulx\Rangle}

\makeatletter
\def\moverlay{\mathpalette\mov@rlay}
\def\mov@rlay#1#2{\leavevmode\vtop{
		\baselineskip\z@skip \lineskiplimit-\maxdimen
		\ialign{\hfil$#1##$\hfil\cr#2\crcr}}}
\makeatother

\newcommand{\re}{\cR_\kk(\ulx)}
\newcommand{\plangle}{\moverlay{(\cr<}}
\newcommand{\prangle}{\moverlay{)\cr>}}
\newcommand{\rx}{\kk\plangle \ulx \prangle}
\newcommand{\rxc}{\C\plangle \ulx \prangle}

\newcommand{\cyceq}{\overset{\operatorname{\scalebox{.5}{cyc}}}{\sim}}
\newcommand{\cyceqs}{\overset{\operatorname{\scalebox{.36}{cyc}}}{\sim}}
\newcommand{\gX}{\Xi}
\newcommand{\gx}{\xi}
\newcommand{\gY}{\Upsilon}

\newtheorem{thm}{Theorem}[section]
\newtheorem{lem}[thm]{Lemma}
\newtheorem{cor}[thm]{Corollary}
\newtheorem{prop}[thm]{Proposition}
\newtheorem{thmA}{Theorem}

\theoremstyle{definition}

\newtheorem{exa}[thm]{Example}

\theoremstyle{remark}
\newtheorem{rem}[thm]{Remark}

\numberwithin{equation}{section}

\begin{document}
	
\setcounter{tocdepth}{3}
\contentsmargin{2.55em} 
\dottedcontents{section}[3.8em]{}{2.3em}{.4pc} 
\dottedcontents{subsection}[6.1em]{}{3.2em}{.4pc}
\dottedcontents{subsubsection}[8.4em]{}{4.1em}{.4pc}

\makeatletter
\newcommand{\mycontentsbox}{%
	{\centerline{NOT FOR PUBLICATION}
		\addtolength{\parskip}{-2.3pt}
		\tableofcontents}}
\def\enddoc@text{\ifx\@empty\@translators \else\@settranslators\fi
	\ifx\@empty\addresses \else\@setaddresses\fi
	\newpage\mycontentsbox\newpage\printindex}
\makeatother

\setcounter{page}{1}

\title[Pencil loci and rational function domains]{Free loci of matrix pencils and domains of noncommutative rational functions}

\author[Igor Klep]{Igor Klep${}^1$}
\address{Igor Klep,  Department of Mathematics, The University of Auckland, New Zealand}
\email{igor.klep@auckland.ac.nz}
\thanks{${}^1$Supported by the Marsden Fund Council of the Royal Society of New Zealand. Partially supported by the Slovenian Research Agency grants P1-0222, L1-4292 and L1-6722.}

\author[Jurij Vol\v{c}i\v{c}]{Jurij Vol\v{c}i\v{c}${}^2$}
\address{Jurij Vol\v{c}i\v{c}, Department of Mathematics, The University of Auckland, New Zealand}
\email{jurij.volcic@auckland.ac.nz}
\thanks{${}^2$Research supported by The University of Auckland Doctoral Scholarship.}

\subjclass[2010]{Primary 15A22, 14P05, 16R30; Secondary 26C15, 16K40, 16N40}
\date{\today}
\keywords{Linear pencil, noncommutative rational function, realization theory, free locus,  real algebraic geometry, hyperbolic polynomial, Kippenhahn's conjecture.}

\begin{abstract}
Consider a monic linear pencil $L(x)=I-A_1x_1-\cdots-A_gx_g$ whose coefficients $A_j$ are $d\times d$ matrices. It is naturally evaluated at $g$-tuples of matrices $X$ using the Kronecker tensor product, which gives rise to its free locus $\fl(L)=\{X:\det L(X)=0\}$. In this article it is shown that the algebras $\cA$ and $\widetilde{\cA}$ generated by the coefficients of two linear pencils $L$ and $\widetilde{L}$, respectively, with equal free loci are isomorphic up to radical, i.e., $\cA/\rad\cA\cong \widetilde{\cA}/\rad\widetilde{\cA}$. Furthermore, $\fl(L)\subseteq \fl(\widetilde{L})$ if and only if the natural map sending the coefficients of $\widetilde{L}$ to the coefficients of $L$ induces a homomorphism $\widetilde{\cA}/\rad\widetilde{\cA}\to \cA/\rad\cA$. Since linear pencils are a key ingredient in studying noncommutative rational functions via realization theory, the above results lead to a characterization of all noncommutative rational functions with a given domain. Finally, a quantum version of Kippenhahn's conjecture on linear pencils is formulated and proved: if hermitian matrices $A_1,\dots,A_g$ generate $M_d(\C)$ as an algebra, then there exist hermitian matrices $X_1,\dots,X_g$ such that $\sum_iA_i\otimes X_i$ has a simple eigenvalue.

The arXiv version of the article contains an appendix presenting an invariant-theoretic
viewpoint and is due independently to Claudio Procesi and \v Spela \v Spenko.
\end{abstract}

\maketitle


\section{Introduction}

Let $\kk$ be a field of characteristic $0$ and let $A_0,A_1,\dots,A_g\in M_d(\kk)$. The formal affine linear combination $L(x)=A_0-A_1x_1-\cdots-A_gx_g$, where 
$x_i$ are freely noncommuting variables, is called an {\bf affine linear pencil}. If $A_0=I_d$ is the $d\times d$ identity matrix, then $L$ is a {\bf (monic) linear pencil}.

Linear pencils are a key tool in matrix theory and numerical analysis (e.g.~the generalized eigenvalue problem), and they frequently appear in algebraic geometry (cf.~\cite{Dol12,Bea}). Linear pencils whose coefficients are symmetric or hermitian matrices give rise to linear matrix inequalities (LMIs), a pillar of control theory, where many classical problems can be converted to LMIs \cite{BEFB,BGM,SIG97}. LMIs also give rise to feasible regions of semidefinite programs in mathematical optimization \cite{WSV12}. In quantum information theory \cite{NC10} and operator algebras \cite{Pau02} hermitian linear pencils are intimately connected to operator spaces and systems, and completely positive maps \cite{HKM}. Lastly, LMIs, linear pencils and their determinants are
studied from a theoretical perspective in real algebraic geometry \cite{HV,Bra,NT,KPV}. 

In this paper we associate
to each linear pencil $L$  its {\bf free (singular) locus} $\fl(L)$, which is defined as the set of all tuples of matrices $X$ over $\kk$ such that
$$L(X)=I\otimes I-\sum_{i=1}^gA_i\otimes X_i$$
is a singular matrix; here $\otimes$ denotes the Kronecker tensor product. We will address the following question: \emph{If $\fl(L)\subseteq\fl(\widetilde{L})$, what can be said about the relation between the coefficients of $L$ and $\widetilde{L}$?}

Our interest in linear pencils originates from their relation with the free skew field of noncommutative rational functions \cite{Ber,Coh,Reu}. Namely, if $\rr$ is a noncommutative rational function that is regular at the origin, then there exists a monic linear pencil $L$ and vectors $\bb,\cc$ over $\kk$ such that
\begin{equation}\label{e:re}
\rr=\cc^t L^{-1}\bb.
\end{equation}
Such presentations of noncommutative rational functions, called realizations, are powerful tools in automata theory \cite{BR}, control theory \cite{BGM,KVV2} and free probability \cite{BMS}. One way of defining noncommutative rational functions is through matrix evaluations of formal noncommutative rational expressions \cite{HMV,KVV1,Vol}. This gives rise to the notion of a domain of a noncommutative rational function, i.e., the set of all matrix tuples where it can be evaluated. While a realization of the form \eqref{e:re} is not unique, there is a canonical, ``smallest'' one $\rr=\cc_0^t L_0^{-1}\bb_0$. The domain of $\rr$ is then the complement of the free locus $\fl(L_0)$ \cite{KVV1}. It is thus natural to ask: (a) When is a noncommutative rational function regular, i.e., defined everywhere? (b) When is the domain of a rational function contained in the domain of another one?
 (c) What can be said about the set of all rational functions with a given domain? 

\subsection{Main results}

Our first main result is a Singularit\"at\-stellensatz for linear pencils explaining when free loci of two linear pencils are comparable. If $L=I-\sum_iA_ix_i$ is a monic pencil of size $d$, let $\cA\subseteq M_d(\kk)$ be the $\kk$-algebra generated by $A_i$. We say that $L$ is minimal if it is of minimal size among all pencils with the same free locus.

\begin{thmA}[Singularit\"atstellensatz]\label{t:intro1}
Let $L$ and $\widetilde{L}$ be monic linear pencils. Then $\fl(L)\subseteq\fl(\widetilde{L})$ if and only if there exists a homomorphism $\widetilde{\cA}/\rad\widetilde{\cA}\to\cA/\rad\cA$ induced by $\widetilde{A}_i\mapsto A_i$.

Moreover, if $L,\widetilde{L}$ are minimal and $\cA,\widetilde{\cA}$ are semisimple, then $\fl(L)=\fl(\widetilde{L})$ if and only if there exists an invertible matrix $P$ such that $\widetilde{A}_i=PA_iP^{-1}$ for all $1\le i\le g$, i.e., the linear pencil $\widetilde{L}$ is a conjugate of $L$.
\end{thmA}

The first part of Theorem \ref{t:intro1} is proved as Theorem \ref{t:main} in Subsection \ref{ss:singsatz}. The second statement appears in Subsection \ref{ss:irr} as Theorem \ref{t:unique}.

Next we combine the Singularit\"at\-stellensatz with the aforementioned realization theory. First we elucidate everywhere-defined noncommutative rational functions. Theorem \ref{t:reg} is an effective version of the following statement.

\begin{thmA}\label{t:intro0}
A regular noncommutative rational function is a noncommutative polynomial.
\end{thmA}

A domain of a noncommutative rational function is co-irreducible if it is not an intersection of larger domains. We say that a noncommutative rational function $\rr$ is irreducible if $\rr=\cc^tL_A^{-1}\bb$, where $L_A$ is a minimal monic pencil and $\cA$ is simple. For every co-irreducible domain $D$ we can find a finite family of linearly independent irreducible functions $\cR(D)$ such that every irreducible function with domain $D$ lies in the linear span of $\cR(D)$. A precise characterization of noncommutative rational functions with a given domain is now as follows.

\begin{thmA}\label{t:intro2}
If a noncommutative rational function $\rr$ is defined at the origin, then its domain equals $D_1\cap\cdots\cap D_s$ for some $s\in\N$ and co-irreducible $D_j$, and $\rr$ is a noncommutative polynomial in $\{x_1,\dots,x_g\}\cup \cR(D_1)\cup\cdots\cup \cR(D_s)$.
\end{thmA}

See Theorem \ref{t:rat} in Subsection \ref{ss:char} for the proof.

Lastly, we apply our techniques to prove the quantum version of Kippenhahn's conjecture \cite{Kip}. The original conjecture was as follows: if hermitian $d\times d$ matrices $H_1$ and $H_2$ generate the whole $M_d(\C)$, then there exist real numbers $\alpha_1$ and $\alpha_2$ such that $\alpha_1H_1+\alpha_2H_2$ has a simple nonzero eigenvalue. While this is false in general \cite{Laf}, we show it is true in a quantum setting.

\begin{thmA}\label{t:intro3}
If $A_1,\dots,A_g\in M_d(\kk)$ generate $M_d(\kk)$ as a $\kk$-algebra, then there exist $n\in\N$ and $X_1,\dots,X_g\in M_n(\kk)$ such that $\sum_iX_i\otimes A_i$ has a nonzero eigenvalue with geometric multiplicity $1$. If $\kk=\C$ and $A_i$ are hermitian, then $X_i$ can also be chosen hermitian.
\end{thmA}

The proof of Theorem \ref{t:intro3} is given in Subsection \ref{ss:kip}.

\subsection{Reader's guide}

The paper is organized as follows. We start by introducing the basic notation and terminology of monic linear pencils, noncommutative rational functions and realizations in Section \ref{sec2}. The inclusion problem for free loci is treated in Section \ref{sec3}. Our main tools are the algebraization trick (Lemma \ref{l:poly}) and the role of the nilradical of the algebra generated by the coefficients of a monic pencil (Proposition \ref{p:nil}). The first part of the Singularit\"atstellensatz is stated in Theorem \ref{t:main}, while Theorem \ref{t:unique} asserts that minimal pencils with the same free locus are unique up to conjugation. The connection between the free locus and the semisimple algebra assigned to a pencil is further investigated in Proposition \ref{p:components} that relates irreducible components of the free locus to the Artin-Wedderburn decomposition of the corresponding semisimple algebra.

In Section \ref{sec4} we apply the preceding results to noncommutative rational functions and their domains. Corollary \ref{c:rat} solves the inclusion problem for domains of noncommutative rational functions in terms of their minimal realizations. As a consequence, Theorem \ref{t:reg} proves that every regular noncommutative rational function (in the sense of being defined everywhere) is a polynomial, which furthermore implies Douglas' lemma for noncommutative rational functions (Corollary \ref{c:douglas}). In Subsection \ref{ss:char} we introduce the notion of co-irreducible domains and derive a precise description of functions with a given domain in Proposition \ref{p:irr} and Theorem \ref{t:rat}.

Finally we focus on symmetric and hermitian pencils, which are ubiquitous in real algebraic geometry \cite{HV,NT} and optimization \cite{HKM,KPV}. Section \ref{sec5} starts by introducing the free real locus assigned to a symmetric or hermitian pencil. Theorem \ref{t:mainsh} is the $*$-analog of the Singularit\"atstellensatz, but instead of noncommutative ring theory its proof crucially relies on properties of hyperbolic polynomials \cite{Gar,Ren} and the real Nullstellensatz \cite{BCR}. Subsection \ref{ss:kip} discusses a relaxation of Kippenhahn's conjecture; its involution-free and hermitian version are resolved by Corollaries \ref{c:kipf} and \ref{c:kipq}, respectively.

\subsubsection*{Added post print}
Appendix \ref{appA}, due independently to Claudio Procesi and
\v Spela \v Spenko, presents an invariant-theoretic viewpoint of some of the
main results of the paper. We thank them for allowing us to include it here.


\linespread{1.125}

\section{Preliminaries}\label{sec2}

In this section we introduce basic notation and the main objects of our study: linear pencils and their (zero) loci, and noncommutative rational functions together with their domains.

\subsection{Basic notation}

Throughout the text let $\kk$ be a field of characteristic 0. If $\ulx=\{x_1,\dots,x_g\}$ is an alphabet, then $\mx$ denotes the free monoid over $\ulx$ and $1\in\mx$ denotes the empty word. Let $\px$ be the free $\kk$-algebra of noncommutative (nc) polynomials. By $\px_+$ we denote its subspace of nc polynomials with zero constant term. For $w\in\mx$ let $|w|\in\N$ denote the length of $w$ and $\mx_h=\{w\in\mx\colon|w|=h\}$. If $\uly$ is another alphabet and $\ulx\cap \uly=\emptyset$, then for $w\in \Langle \ulx\cup \uly\Rangle$ let $|w|_{\uly}$ denote the number of occurrences of elements from $\uly$ in $w$. Lastly, $\cyceq$ denotes the cyclic equivalence relation on words, i.e., $w_1\cyceq w_2$ if and only if there exist words $u$ and $v$ such that $w_1=uv$ and $w_2=vu$. Equivalently, $w_1$ is a cyclic permutation of $w_2$.

\subsubsection{Free locus of a linear pencil}

If $A_1,\dots,A_g\in M_d(\kk)$, then
$$L=I-\sum_{i=1}^g A_ix_i\in M_d(\px)$$
is called a {\bf monic linear pencil} of size $d$. We write $L=L_A$ if we want to emphasize which coefficients appear in $L$. The evaluation of $L$ at a point $X=(X_1,\dots,X_g)\in M_n(\kk)^g$ is defined using the (Kronecker) tensor product
$$L(X)=I\otimes I-\sum_{i=1}^g A_i\otimes X_i\in M_{nd}(\kk).$$
The {\bf free (singular) locus} of $L$ is the set
\begin{equation}\label{e:sing}
\fl(L)=\bigcup_{n\in\N}\fl_n(L),\qquad \text{where}\quad \fl_n(L)=\left\{X\in M_n(\kk)^g\colon\det(L(X))=0 \right\}.
\end{equation}
Clearly, each $\fl_n(L)$ is an algebraic subset of $M_n(\kk)^g$.

\subsection{Noncommutative rational functions}\label{ss:rat}

We introduce noncommutative rational functions using matrix evaluations of formal rational expressions following \cite{HMV,KVV2}. Originally they were defined ring-theoretically, cf. \cite{Ber,Coh}. A syntactically valid combination of nc polynomials, arithmetic operations $+,\ \cdot,\ {}^{-1}$ and parentheses $(,)$ is called a {\bf noncommutative (nc) rational expression}. The set of all nc rational expressions is 
denoted $\re$. For example, $(1+x_3^{-1}x_2)+1$, $x_1+(-x_1)$ and $0^{-1}$ are elements of $\re$.

Every polynomial $f\in\px$ can be naturally evaluated at a point $A\in M_n(\kk)^g$ by replacing $x_j$ with $A_j$ and $1$ with $I$; the result is $f(A)\in M_n(\kk)$. We can naturally extend evaluations of nc polynomials to evaluations of nc rational expressions. Given $r\in\re$, then $r(A)$ is defined in the obvious way if all inverses appearing in $r$ exist at $A$. Let $\dom_n r$ be the set of all $A\in M_n(\kk)$ such that $r$ is defined at $r$. Then the {\bf domain} of a nc rational expression $r$ is
$$\dom r=\bigcup_{n\in\N} \dom_n r$$
and $r$ is {\bf non-degenerate} if $\dom r\neq\emptyset$.

On the set of all non-degenerate nc rational expressions we define an equivalence relation $r_1\sim r_2$ if and only if $r_1(A)=r_2(A)$ for all $A\in\dom r_1\cap\dom r_2$. Then {\bf noncommutative (nc) rational functions} are the equivalence classes of non-degenerate nc rational expressions. By \cite[Proposition 2.1]{KVV2} they form a skew field denoted $\rx$. It is the universal skew field of fractions of $\px$ \cite[Section 4.5]{Coh}. For $\rr\in\rx$ let $\dom_n\rr$ be the union of $\dom_n r$ over all representatives $r\in\re$ of $\rr$. Then the {\bf domain} of a nc rational function $\rr$ is
$$\dom \rr=\bigcup_{n\in\N}\dom_n \rr.$$

\subsubsection{Realizations}

Let $\rx_0\subset \rx$ denote the local subring of nc rational functions that are regular at the origin:
$$\rx_0=\{\rr\in\rx\colon 0\in\dom\rr\}.$$
A very powerful tool for operating with elements from $\rx_0$ is realization theory. If $\rr\in\rx_0$, then there exist $d\in\N$, $\cc,\bb\in\kk^d$ and a monic linear pencil $L$ of size $d$ such that
$$\rr=\cc^tL^{-1}\bb.$$
Such a triple $(\cc,L,\bb)$ is called a {\bf realization of $\rr$} of size $d$. We refer to \cite{BR,HMV} for a good exposition on classical realization theory; also see \cite{Vol} for realizations about arbitrary matrix points which can consequently be applied to arbitrary nc rational functions.

Let us fix $\rr\in\rx_0$. In general, $\rr$ admits various realizations. A realization of $\rr$ whose size is minimal among all realizations of $\rr$ is called {\bf minimal}. The following facts comprise the importance of minimal realizations.
\begin{enumerate}
	\item Minimal realizations are unique up to similarity by \cite[Theorem 2.4]{BR}. That is, if $(\cc,L,\bb)$ and $(\cc',L',\bb')$ are minimal realizations of $\rr$ of size $d$, then there exists $P\in\GL_d(\kk)$ such that $\cc'= P^{-t}\cc$, $L'=P LP^{-1}$ and $\bb'=P\bb$.
	\item If $(\cc,L,\bb)$ is a minimal realization of $\rr$, then
	$$\dom\rr=\bigcup_{n\in\N}\{X\in M_n(\kk)^g\colon \det(L(X))\neq0\}=\fl(L)^c$$
	by \cite[Theorem 3.1]{KVV1} and \cite[Theorem 3.10]{Vol1}.
	\item By \cite[Section II.3]{BR}, there is an efficient algorithm that provides us with a minimal realization of $\rr$.
\end{enumerate}
Hence the domain of a nc rational function regular at 0 can be described as a complement of a free locus. Similar result also holds for an arbitrary rational function \cite[Corollary 5.9]{Vol}.


\section{Inclusion problem for free loci}\label{sec3}

In this section we investigate when free loci of two linear pencils are comparable. The main results are the Singularit\"atstellens\"atze \ref{t:main} and \ref{t:unique}. Theorem \ref{t:main} shows that inclusion of free loci is equivalent to the existence of a homomorphism between semisimple algebras associated to the two pencils. Theorem \ref{t:unique} proves that (under natural minimality assumptions) two pencils with the same free locus are similar, i.e., one is a conjugate of the other. Our main technical ingredient in the proofs is the algebraization trick of Subsection \ref{ss:algtrick}, which relates properties of a linear pencil to properties of the matrix algebra generated by the coefficients of the pencils.

\subsection{Algebraization trick}\label{ss:algtrick}

Lemma \ref{l:poly} will be used repeatedly in the sequel to pass from a pencil $L_A$ to the $\kk$-algebra $\cA$ generated by matrices $A_1,\dots,A_g$.

\begin{lem}\label{l:poly}
	For every $f\in\px_+$ and $X_i,Y\in M_n(\kk)$ there exist $N\in\N$ and $X'_i\in M_N(\kk)$ such that
	\begin{equation}\label{e:ker}
	\dim\ker(L_A(X)-f(A)\otimes Y)=\dim\ker L_A(X')
	\end{equation}
	for all $d\in\N$ and $A_i\in M_d(\kk)$.
\end{lem}

\begin{proof}
We prove a slightly stronger statement: for every $f\in\px_+$, $h\in\N$ and $X_1,\dots,X_g,Z_1,\dots,Z_h,Y\in M_n(\kk)$ there exist $N\in\N$ and $X_1',\dots,X_g',Z_1',\dots,Z_h'\in M_N(\kk)$ such that
\begin{equation}\label{e:ker1}
\dim\ker(L_{A,C}(X,Z)-f(A)\otimes Y)=\dim\ker L_{A,C}(X',Z')
\end{equation}
for all $d\in\N$ and $A_1,\dots,A_g,C_1,\dots,C_h\in M_d(\kk)$, where $L_{A,C}(x,z)=I-\sum_iA_ix_i-\sum_kC_kz_k$.

First observe that
$$\begin{pmatrix}u_1\\ u_2\end{pmatrix}\in \ker\begin{pmatrix}I& M_1\\ M_2& M\end{pmatrix} \quad \iff \quad
u_2\in \ker(M-M_2M_1),\ u_1=-M_1u_2$$
for all matrices $M,M_1,M_2$ of consistent sizes and therefore
\begin{equation}\label{e:e0}
\dim\ker(M-M_2M_1)=\dim\ker\begin{pmatrix}I& M_1\\ M_2& M\end{pmatrix}.
\end{equation}
If the stronger statement holds for $f$ and $g$, then it also holds for $\alpha f+\beta g$ for $\alpha,\beta\in\kk$ since
\begin{align*}
	&\dim\ker(L_{A,C}(X,Z)-(\alpha f+\beta g)(A)\otimes Y) \\
	=&\dim\ker(L_{A,C,f(A)}(X,Z,\alpha Y)-g(A)\otimes \beta Y)\\
	=&\dim\ker L_{A,C,f(A)}(X',Z',Y')\\
	=&\dim\ker(L_{A,C}(X',Z')-f(A)\otimes Y')\\
	=&\dim\ker L_{A,C}(X'',Z'')
\end{align*}
for appropriate $X_i',Z_j',Y'\in M_{N_1}(\kk)$ and $X_i'',Z_j''\in M_{N_2}(\kk)$ that exist by assumption. Hence it suffices to establish the statement for $f=w\in \mx\setminus\{1\}$. We prove \eqref{e:ker1} by induction on $|w|$. The case $|w|=1$ is clear, so assume that \eqref{e:ker1} holds for all words of length $\ell\ge1$. If $w=x_jv$ for $|v|=\ell$, then
\begin{align*}
	&\dim\ker(L_{A,C}(X,Z)-w(A)\otimes Y)\\
	=&\dim\ker\begin{pmatrix}I\otimes I & -v(A)\otimes I \\ -A_j\otimes Y & L_{A,C}(X,Z)\end{pmatrix}\\
	=&\dim\ker\left(
	L_{A,C}\left( \begin{pmatrix}0 & 0\\0 & X\end{pmatrix},\begin{pmatrix}0 & 0\\0 & Z\end{pmatrix}\right)
	-A_j\otimes \begin{pmatrix}0 & 0\\Y & 0\end{pmatrix}
	-v(A)\otimes \begin{pmatrix}0 & I\\0 & 0\end{pmatrix}
	\right)\\
	=&\dim\ker L_{A,C}(X',Z')
\end{align*}
for some $X_i',Z_j'\in M_N(\kk)$ by \eqref{e:e0}, conjugation with an invertible matrix, and the induction hypothesis.
\end{proof}

As it follows from the proof, the number $N$ in the statement of Lemma \ref{l:poly} can be bounded by a function which is polynomial in $n$ and exponential in the degree of $f$ and number of terms in $f$.

\begin{cor}\label{c:poly}
	If $\fl(L_A)\subseteq\fl(L_B)$, then $\fl(L_A-f(A)y)\subseteq\fl(L_B-f(B)y)$ for every $f\in\px_+$.
\end{cor}

\begin{proof}
	If $(X,Y)\in \fl(L_A-f(A)y)$, let $X'$ be as in Lemma \ref{l:poly}. Then $X'\in\fl(L_A)$ and therefore $X'\in \fl(L_B)$ by assumption, so $(X,Y)\in\fl (L_B-f(B)y)$ since the choice of $X'$ is independent of the pencils $L_A$ and $L_B$.
\end{proof}

\subsection{Jointly nilpotent coefficients}

The question whether an evaluation of a pencil $L_A(x)$ is invertible might be independent of some of the variables in $\ulx$. In this subsection we show that in this case their corresponding coefficients in $L_A$ generate a nilpotent ideal. Moreover, we provide explicit polynomial bounds originating from the theory of polynomial trace identities \cite{Pro} and bounds on lengths of generating sets of matrix subalgebras \cite{Pap} to check whether this happens.

Let $\cA$ be a (possibly non-unital) finite-dimensional $\kk$-algebra. If $S\subseteq \cA$ is its generating set, then we define the {\bf length} of $S$ as
$$
\ell(S)=\min\left\{l\in\N\colon \bigcup_{j=1}^lS^j \ \text{linearly spans }\cA\right\}.
$$
Here $S^j$ is the set of all products of $j$ elements of $S$. Denote
$$\lambda(d)=\left\{\begin{array}{ll}
1 & d=1, \\
\left\lceil d\sqrt{\frac{2d^2}{d-1}+\frac{1}{4}}+\frac{d}{2}-2 \right\rceil& d\ge2.
\end{array}\right.$$
By \cite[Theorem 3.1]{Pap} we have $\ell(S)\le \lambda(d)\approx \sqrt2 d^{3/2}$ for every generating set $S$ of $\cA\subseteq M_d(\kk)$.

In the sequel we also require the following notion. For $g,n\in\N$ let
$$\kk\left[\ulxx\right]=\kk\left[\gx_{\imath\jmath}^{(i)}\colon 1\le i\le g,\,1\le\imath,\jmath\le n\right]$$
be the ring of polynomials in $gn^2$ commutative indeterminates. The distinguished matrices
$$\gX_i=\left(\gx_{\imath\jmath}^{(i)}\right)_{\imath\jmath}\in M_n\left(\kk\left[\ulxx\right]\right)$$
are called the {\bf generic $n\times n$ matrices} \cite[Section 6.7]{Bre}.

\begin{prop}\label{p:nil}
	Let $\cA\subseteq M_d(\kk)$ be the $\kk$-algebra generated by
	$$A_1,\dots,A_g,N_1,\dots,N_h\in M_d(\kk),$$
	and let $\cN\subseteq \cA$ be the ideal generated by $N_1,\dots,N_h$. If $m\ge \lambda(d)$ and
	\begin{equation}\label{e:nil}
	\det\left(L_A(X)-\sum_jN_j\otimes Y_j\right)=\det(L_A(X))
	\end{equation}
	holds for all $X_i,Y_j\in M_m(\kk)$, then $\cN$ is a nilpotent ideal in $\cA$.
	
	Conversely, if $\cN$ is nilpotent, \eqref{e:nil} holds for all $X_i,Y_j\in M_n(\kk)$ and $n\in\N$.
\end{prop}

\begin{proof}
Assume \eqref{e:nil} holds. Let $\gX_i$ be generic $m\times m$ matrices. As a matrix over the ring of formal power series $\kk[[\ulxx]]$, $L_A(\gX)$ is invertible and
\begin{equation}\label{e:neumann}
L_A(\gX)^{-1}=\sum_{w\in \mx} w(A)\otimes w(\gX)
\end{equation}
by the Neumann series expansion. Then \eqref{e:nil} implies
$$\det\left(I\otimes I-\left(\sum_jN_j\otimes Y_j\right)L_A(\gX)^{-1}\right)=1$$
for every $Y_i\in M_m(\kk)$. In particular, for
$$p(t)=\det\left(I\otimes I-t\left(\sum_jN_j\otimes Y_j\right)L_A(\gX)^{-1}\right)\in \kk[[\ulxx]][t]$$
we have $p(t)=1$, so $(\sum_jN_j\otimes Y_j)L_A(\gX)^{-1}$ does not have nonzero eigenvalues and is therefore a nilpotent matrix. Hence $(\sum_jN_j\otimes \gY_j)L_A(\gX)^{-1}$ is nilpotent, where $\gY_j$ are generic $m\times m$ matrices, so
\begin{align*}
	0
	&=\tr\left(\left(\left(\sum_jN_j\otimes \gY_j\right)\left(\sum_{w\in \mx} w(A)\otimes w(\gX)\right)\right)^\ell\right) \\
	&=\tr\left(\sum_{w\in \uly\!\mx\cdots \uly\!\mx,\,|w|_Y=\ell}w(A,N)\otimes w(\gX,\gY)\right) \\
	&=\sum_{w\in \uly\!\mx\cdots \uly\!\mx,\,|w|_{\uly}=\ell}\tr(w(A,N))\tr(w(\gX,\gY)) \\
	&=\sum_{[w]\in \Langle \ulx\cup \uly \Rangle\!/\cyceqs,\,|w|_{\uly}=\ell} \mu_w\tr(w(A,N))\tr(w(\gX,\gY))
\end{align*}
for every $\ell\in\N$, where $0<\mu_w=|[w]\cap \uly\!\mx\cdots \uly\!\mx|$ for $w\in \Langle \ulx\cup \uly\Rangle$ with $|w|_{\uly}=\ell$. Here $[w]$ denotes the equivalence class of $w$ with respect to $\cyceq$. For every $h\in\N$, the pure trace polynomial
$$p_h=\sum_{[w]\in \Langle \ulx\cup \uly \Rangle_h\!/\cyceqs,\,|w|_{\uly}>0} \mu_w\tr(w(A,N))\tr(w)$$
of degree $h$ therefore vanishes on all tuples of $m\times m$ matrices. By \cite[Theorem 4.5, Proposition 8.3]{Pro} we have $p_h=0$ for all $h\le m$. Therefore $\tr(w(A,N))=0$ for every $w\in \Langle \ulx\cup \uly\Rangle_h$ with $|w|_{\uly}>0$ and $h\le m$. Since $m\ge \lambda(d)$, the discussion above implies that
$$\{w(A,N)\colon 1\le|w|\le m,\,|w|_{\uly}>0\}$$
linearly spans $\cN$. Therefore $\tr(w(A,N))=0$ for every $w\in \Langle \ulx\cup \uly \Rangle$ with $|w|_{\uly}>0$, hence $\cN\subseteq \cA$ is a nilpotent ideal.

Conversely, suppose $\cN$ is nilpotent. Let $\kkb$ be the algebraic closure of $\kk$. Burnside's theorem on the existence of invariant subspaces \cite[Corollary 5.23]{Bre} applied to $\cA\otimes_{\kk}\kkb$ yields a vector space decomposition
\begin{equation}\label{e:decomp}
\kkb^d=U_1\oplus \cdots \oplus U_s
\end{equation}
such that $\cA U_k\subseteq U_1\oplus\cdots\oplus U_k$ and $\pi_k(\cA\otimes_{\kk}\kkb)\iota_k$ is either $\{0\}$ or $\End_{\kkb}(U_k)$, where $\iota_k:U_k\to \kkb^d$ and $\pi_k:\kkb^d\to U_k$ are the canonical inclusion and projection, respectively. We claim that $\cN U_k\subseteq U_1\oplus\cdots\oplus U_{k-1}$; indeed, if $\pi_k(\cN U_k)\cap U_k\neq\{0\}$, then the simplicity of $\End_{\kkb}(U_k)$ implies $I\in\pi_k(\cN\otimes_{\kk}\kkb)\iota_k$, which is a contradiction since $\cN\otimes_{\kk}\kkb$ is nilpotent.

Because the determinant of a block-upper-triangular matrix is equal to the product of determinants of its diagonal blocks, the decomposition \eqref{e:decomp} and the structure of the Kronecker product imply 
$$\det\left(L_A(X)-\sum_jN_j\otimes Y_j\right)=\det(L_A(X))$$
for all $X_i,Y_j\in M_n(\kk)$ and all $n\in\N$.
\end{proof}

\begin{cor}
If $L$ is a monic linear pencil, then $L(X)$ is invertible for all matrix tuples $X$ if and only if the coefficients of $L$ are jointly nilpotent.
\end{cor}

Of course, just assuming that $L(\alpha)$ is invertible for all scalar tuples $\alpha\in\kk^g$ does not imply that coefficients of $L$ are jointly nilpotent. For example, if
$$L=I-\begin{pmatrix}0&1&0\\0&0&0\\1&0&0\end{pmatrix}x_1-\begin{pmatrix}0&0&-1\\1&0&0\\0&0&0\end{pmatrix}x_2,$$
then every linear combination of the coefficients of $L$ is nilpotent and hence $\fl_1(L)=\emptyset$, but the coefficients are not jointly nilpotent. For an investigation of linear spaces of nilpotent matrices see e.g. \cite{MOR}.

\subsection{Singularit\"atstellensatz}\label{ss:singsatz}

This subsection contains the main result of this section. Theorem \ref{t:main} translates the inclusion between two free loci $\fl(L_A)\subseteq\fl(L_B)$ into a purely algebraic statement about algebras generated by the matrices $A_i$ and $B_i$.

For a (possibly non-unital) finite-dimensional $\kk$-algebra $\cR$ let $\rad \cR$ be its largest nilpotent ideal; we call it the (nil)radical of $\cR$. If $\cR\neq\rad\cR$, then $\cR/\rad\cR$ is semiprime and hence semisimple \cite[Theorem 2.65]{Bre}. Note that such a ring contains a multiplicative identity $\one$ and that an epimorphism of unital rings preserves the identity.

\begin{rem}\label{r:const}
Let $N\in M_n(\kk)$ and consider $p=\det(I-tN)\in\kk[t]$. Then $N$ is nilpotent if and only if $p=1$. This is furthermore equivalent to
$$p(T)=\det\left(I\otimes I-T\otimes N\right)\neq0$$
for all $T\in M_n(\kk)$ because the companion matrix associated to $p$ is of size $\deg p\le n$. If $\kk$ is an algebraically closed field or a real closed field, then it of course suffices to test $p(T)\neq0$ for all $T\in \kk$ or $T\in M_2(\kk)$, respectively.
\end{rem}

\begin{thm}[Singularit\"atstellensatz]\label{t:main}
	Let $\cA\subseteq M_d(\kk)$ be the subalgebra generated by $A_1,\dots, A_g$ and let $\cB\subseteq M_e(\kk)$ be the subalgebra generated by $B_1,\dots, B_g$. Then $\fl(L_A)\subseteq\fl(L_B)$ if and only if there exists a homomorphism of $\kk$-algebras $\cB / \rad\cB \to \cA / \rad\cA$ induced by $B_i\mapsto A_i$.
\end{thm}

\begin{proof}
($\Rightarrow$) It suffices to prove that for every $f\in\px_+$, $f(B)\in\rad\cB$ implies $f(A)\in\rad\cA$. If $f(B)$ generates a nilpotent ideal in $\cB$, then 
\begin{equation}\label{e:e1}
\fl(L_A-f(A)y)\subseteq \fl(L_B-f(B)y)=\fl(L_B-0\cdot y)
\end{equation}
by Corollary \ref{c:poly} and Proposition \ref{p:nil}.

For $n\in\N$ let $\gX_i,\gY$ be $n\times n$ generic matrices and let
$$p=\det\left(L_A(\gX) -f(A)\otimes \gY\right).$$
Suppose there exist $1\le \imath_0,\jmath_0\le n$ such that $\frac{\partial p}{\partial t}\neq0$, where $t=(\gY)_{\imath_0\jmath_0}$. Because $\kk$ is infinite, there exist $X_i\in M_n(\kk)$  and $\alpha_{\imath\jmath}\in\kk$ for all $\imath\neq\imath_0$ and $\jmath\neq\jmath_0$ such that 
$$\det(L_B(X))\neq0,\quad \frac{\partial p}{\partial t}(X,\alpha,t)\neq0.$$
Let $q=p(X,\alpha,t)\in\kk[t]$; since $q$ is non-constant polynomial of degree at most $nd$, there exists $T\in M_{nd}(\kk)$ such that $q(T)=0$ by Remark \ref{r:const}. Now let $Y'\in M_{n^2d}(\kk)$ be a block $n\times n$ matrix such that its $(\imath,\jmath)$-block equals $T$ if $\imath=\imath_0$ and $\jmath=\jmath_0$, and $\alpha_{\imath\jmath} I$ otherwise. Then
$$\det(L_A(X\otimes I) -f(A)\otimes Y')=0,$$
which contradicts \eqref{e:e1} since $\det(L_B(X\otimes I))\neq0$.

Hence the free locus of $I-\sum_iA_ix_i -f(A)y$ does not depend on $y$ and so $\fl(L_A-f(A)y)=\fl(L_A-0\cdot y)$. Therefore $f(A)$ generates a nilpotent ideal in $\cA$ by Proposition \ref{p:nil}.

($\Leftarrow$) Let $a_i$ and $b_i$ be equivalence classes of $A_i$ and $B_i$ in $\cA/\rad\cA$ and $\cB/\rad\cB$, respectively, and assume there is a homomorphism $\phi:\cB/\rad\cB\to\cA/\rad\cA$ satisfying $\phi(b_i)=a_i$. Suppose $\det(L_B(X))\neq0$ for $X\in M_n(\kk)^g$. Then there exists $p\in\kk[t]$, $p(0)=0$, such that $p\left(I\otimes I-\sum_iB_i\otimes X_i\right)=I\otimes I$ by the Cayley-Hamilton theorem. Let $q(t)=p(1-t)-p(1)$; then $q(0)=0$ and
$q\left(\sum_iB_i\otimes X_i\right)=(1+q(1))I\otimes I$. If $q(1)\neq-1$, then $I\otimes I\in M_n(\cB)$ and hence $I\in\cB$, so
$$q\left(\sum_ib_i\otimes_{\kk} X_i\right)=(1+q(1))\one_{\cB/\rad\cB}\otimes_{\kk} I\in (\cB/\rad\cB)\otimes_{\kk} M_n(\kk);$$
On the other hand, if $q(1)=-1$, then $q\left(\sum_iB_i\otimes X_i\right)=0$ and so $q\left(\sum_ib_i\otimes_{\kk} X_i\right)=0$.
Since $\phi(\one_{\cB})=\one_{\cA}$, both cases imply
$$q\left(\sum_ia_i\otimes_{\kk} X_i\right)=(1+q(1))\one_{\cA/\rad\cA}\otimes_{\kk} I\in (\cA/\rad\cA)\otimes_{\kk} M_n(\kk).$$
Consequently $q\left(\sum_iA_i\otimes X_i\right)=(1+q(1))I\otimes I+N$ for some $N\in M_n(\rad\cA)$ and therefore
$p\left(I\otimes I-\sum_iA_i\otimes X_i\right)=I\otimes I+N$, so $\det(L_A(X))\neq0$ since $N$ is nilpotent. Thus $\fl(L_A)\subseteq\fl(L_B)$.
\end{proof}

\begin{rem}\label{r:bounds}
	Let $L_1$ and $L_2$ be monic linear pencils of sizes $d_1$ and $d_2$, respectively. By Proposition \ref{p:nil} and proofs of Lemma \ref{l:poly} and Theorem \ref{t:main} one can derive deterministic bounds on size of matrices $X_1,\dots,X_g$ for checking $\fl(L_1)\subseteq \fl(L_2)$ that are exponential in $g$ and $\max\{d_1,d_2\}$.
\end{rem}

From here on we write $\cA$ (resp. $\cB$) for the (possibly non-unital) $\kk$-algebra generated by the coefficients $A_1,\dots,A_g$ (resp. $B_1,\dots,B_g$) of the pencil $L_A$ (resp. $L_B$).

\begin{cor}\label{c:equal}
	Let the notation be as in Theorem \ref{t:main}. Then $\fl(L_A)=\fl(L_B)$ if and only if there exists an isomorphism $\cA / \rad\cA \to \cB / \rad\cB$ induced by $A_i\mapsto B_i$.
\end{cor}

The validity of $\fl(L_A)\subseteq\fl(L_B)$ can now be effectively tested. Using probabilistic algorithms for finding the radical of a finite-dimensional algebra (see e.g. \cite{CIW}) we first reduce the problem to the case where $\cA$ and $\cB$ are semisimple. Then we find $\ell\le \lambda(\max\{d,e\})$ such that $\{w(A)\colon 1\le|w|\le \ell\}$ linearly spans $\cA$ and $\{w(B)\colon 1\le|w|\le \ell\}$ linearly spans $\cB$. Next, we determine the linear relations between the elements of $\{w(B)\colon 1\le|w|\le \ell+1\}$. Finally we check whether they are also satisfied by $\{w(A)\colon 1\le|w|\le \ell+1\}$.

\subsection{Irreducible free loci}\label{ss:irr}

In this subsection we discuss irreducible components of free loci and how they correspond to the Artin-Wedderburn decomposition of the semisimple algebra $\cA/\rad\cA$ assigned to a pencil $L_A$.

\begin{rem}\label{r:simple}
	Let $\cA$ be a finite-dimensional simple $\kk$-algebra. Then $\cA\cong M_m(\Delta)$ for some finite-dimensional division $\kk$-algebra $\Delta$. Up to isomorphism there is exactly one simple unital left $\cA$-module, namely $\Delta^m$, and every unital left $\cA$-module is isomorphic to a direct sum of copies of $\Delta^m$. Let $\delta=m\dim_{\kk}\Delta$; then there exists an irreducible representation $\rho:\cA\to M_{\delta}(\kk)$, which is unique up to conjugation by the Skolem-Noether theorem \cite[Theorem 4.48]{Bre}, and every representation of $\cA$ factors through it.
\end{rem}

We will also use the following refinement of the Skolem-Noether theorem.

\begin{lem}\label{l:SN}
	For $1\le j\le s$ let $\rho_j:\cA^{(j)}\to M_{d_j}(\kk)$ be an irreducible representation of a simple $\kk$-algebra $\cA^{(j)}$. If $\iota: \cA^{(1)}\times\cdots\times \cA^{(s)}\to M_d(\kk)$ is a unital embedding, then there exists $P\in\GL_d(\kk)$ such that
	$$P\iota(a) P^{-1} = (I\otimes\rho_1(a))\oplus\cdots\oplus (I\otimes\rho_s(a)) \in M_d(\kk)
	\qquad \forall a \in\cA^{(1)}\times\cdots\times \cA^{(s)}.$$
\end{lem}

\begin{proof}
Consider vector subspaces $U_j=\im \iota(\one_{\cA^{(j)}})$ for $1\le j\le s$; it is easy to check that $\kk^d=U_1\oplus\cdots\oplus U_s$, $\iota(\cA^{(j)}) U_j\subseteq U_j$ and $\iota(\cA^{(j)}) U_{j'}=0$ for $j'\neq j$. Hence we have a unital embedding $\cA^{(j)}\to \End_{\kk}(U_j)$. By the Skolem-Noether theorem there exists $P_j\in\End_{\kk}(U_j)$ such that
$$P_j\iota|_{\cA^{(j)}}(a_j)P_j^{-1}=I\otimes\rho_j(a_j)$$
for all $a_j\in\cA^{(j)}$. If $P_0\in\GL_d(\kk)$ is the transition matrix corresponding to the decomposition $\kk^d=U_1\oplus\cdots\oplus U_s$, then let $P=P_0(P_1\oplus\cdots\oplus P_s)$.
\end{proof}

A pencil $L$ is {\bf minimal} if it is of the smallest size among all pencils whose free loci are equal to $\fl(L)$. (Note: (i) a pencil of a minimal realization is not necessarily minimal; (ii) a realization with a minimal pencil is not necessarily minimal.) A minimal pencil $L_A$ is {\bf irreducible} if $\cA$ is simple.

\begin{thm}\label{t:unique}
	Let $L_A$ and $L_B$ be minimal pencils of size $d$ and assume that $\cA$ and $\cB$ are semisimple. Then $\fl(L_A)=\fl(L_B)$ if and only if there exists $P\in\GL_d(\kk)$ such that $B_i=PA_iP^{-1}$ for $i=1,\dots,g$.
\end{thm}

\begin{proof}
If $\fl(L_A)=\fl(L_B)$, then $d=e$ by minimality. As elements of $M_d(\kk)$, $\one_{\cA}$ and $\one_{\cB}$ are idempotents. If for example $\one_{\cA}$ were a nontrivial idempotent, then the restriction and projection of matrices $A_i$ to subspace $\im\one_{\cA}$ would yield a smaller pencil with the same free locus, which contradicts the minimality assumption. Hence $\one_{\cA}=\one_{\cB}=I$. By Corollary \ref{c:equal} and semisimplicity we have
$$\cA\xleftarrow{\phi_1} \cC^{(1)}\times\cdots\times \cC^{(s)}\xrightarrow{\phi_2}\cB$$
for some simple algebras $\cC^{(j)}$ and isomorphisms $\phi_1,\phi_2$ satisfying $\phi_2\phi_1^{-1}(A_i)=B_i$. Let $\rho_j:\cC^{(j)}\to M_{d_j}(\kk)$ be an irreducible representation of $\cC^{(j)}$. By Lemma \ref{l:SN} and minimality there exist $P_1,P_2\in\GL_d(\kk)$ such that
$$P_1\phi_1(c) P_1^{-1} = \rho_1(c)\oplus\cdots\oplus \rho_s(c) = P_2\phi_2(c)P_2^{-1}$$
for all $c\in \cC^{(1)}\times\cdots\times \cC^{(s)}$. Therefore $P=P_2^{-1}P_1$ satisfies $B_i=PA_iP^{-1}$.
\end{proof}

A free locus is {\bf irreducible} if it is nonempty and not a union of two smaller free loci. Note that $\fl(L_1\oplus L_2)=\fl(L_1)\cup\fl(L_2)$.

\begin{prop}\label{p:components}\hspace{0em}
\begin{enumerate}[label={\rm(\roman*)}]
	\item If $\cA/\rad \cA$ is isomorphic to the product of $s$ simple algebras, then $\fl(L_A)$ has exactly $s$ irreducible components.
	\item Every irreducible free locus equals $\fl(L)$ for some irreducible $L$.
\end{enumerate}
\end{prop}

\begin{proof}
(i) Let $\phi:\cA/\rad \cA\to\cA^{(1)}\times\cdots\times \cA^{(s)}$ be an isomorphism to a direct product of simple algebras $\cA^{(j)}$. Let $A_i^{(j)}$ be the image of $A_i$ under the homomorphism $\cA\to\cA^{(j)}\to M_{d_j}(\kk)$, where $\cA^{(j)}\to M_{d_j}(\kk)$ is an arbitrary faithful representation. Then Corollary \ref{c:equal} yields
$$\fl(L_A)=\fl(L_{A^{(1)}\oplus\cdots\oplus A^{(s)}})=\fl(L_{A^{(1)}})\cup\cdots\cup \fl(L_{A^{(s)}}).$$
Also, $j_1\neq j_2$ implies $\fl(L_{A^{(j_1)}})\neq \fl(L_{A^{(j_2)}})$. Otherwise there would exist an isomorphism $\psi:\cA^{(j_1)}\to \cA^{(j_2)}$ given by $A^{(j_1)}_i\mapsto A^{(j_2)}_i$. If $\phi_{j_1}=\pi_{j_i}\phi$ and $\phi_{j_1}=\pi_{j_i}\phi$, where $\pi_j:\cA^{(1)}\times\cdots\times \cA^{(s)}\to \cA^{(j)}$ is the natural projection, then $\phi_{j_2}=\psi\phi_{j_1}$ and so $\phi_{j_1}(f(A))=0$ if and only if $\phi_{j_2}(f(A))=0$ for every $f\in\px_+$, which contradicts the surjectivity of $\phi$. Hence it suffices to prove that $\fl(L_A)$ is irreducible if $\cA$ is simple.

Suppose $\fl(L_A)=\fl(L_{A'})\cup \fl(L_{A''})=\fl(L_{A'\oplus A''})$ and let $\cB$ be the algebra generated by matrices $A'_i\oplus A''_i$. Then $\cA\cong \cB/\rad\cB$ by Corollary \ref{c:equal}, hence there is an embedding
$$\cA\hookrightarrow (\cA'\times \cA'')/\rad (\cA'\times \cA'')=(\cA'/\rad \cA')\times (\cA''/\rad \cA'')$$
such that the induced homomorphisms $\cA\to \cA'/\rad \cA'$ and $\cA\to \cA''/\rad \cA''$ are surjective. Since $\cA$ is simple, the induced map $\cA\to \cA'/\rad \cA'$ is trivial or injective. In the latter case $\cA\cong\cA'/\rad \cA'$ via $A_i\mapsto A'_i$, so Theorem \ref{t:main} implies $\fl(L_A)=\fl(L_{A'})$. Since $\fl(L_A)\neq\emptyset$, Theorem \ref{t:reg} implies that $\cA'/\rad \cA'$ and $\cA''/\rad \cA''$ cannot be both trivial, so we conclude that $\fl(L_A)=\fl(L_{A'})$ or $\fl(L_A)=\fl(L_{A''})$. Therefore $\fl(L_A)$ is irreducible.

(ii) If $\fl(L_B)$ is irreducible, then $\cB/\rad \cB=\cA$ is a simple algebra by (i). By Remark \ref{r:simple} there exists an irreducible representation $\cA\to M_d(\kk)$. Set $A_i$ to be the image of $B_i$ under the homomorphism $\cB\to\cA\to M_d(\kk)$; then $L_A$ is the desired irreducible pencil.
\end{proof}

The radical of a finite-dimensional algebra and the Wedderburn decomposition of a semisimple algebra can be computed using probabilistic algorithms with polynomial complexity \cite{FR,Ebe}. By Proposition \ref{p:components} we can therefore efficiently determine irreducible components of a free locus. 
In a forthcoming paper it will be shown that if $\kk$ is algebraically closed and $\fl(L)$ is an irreducible free locus, then $\fl_n(L)$ is an irreducible algebraic set in $M_n(\kk)^g$ for sufficiently large $n\in\N$.


\section{Domains of noncommutative rational functions regular at the origin}\label{sec4}

In this section we shall explain how our results on free loci pertain to domains of nc rational functions. The main results are Corollary \ref{c:rat} and Theorem \ref{t:rat}. While Corollary \ref{c:rat} relates the inclusion of domains of nc rational functions to homomorphisms between the algebras associated to their minimal realizations, Theorem \ref{t:rat} analyzes the precise structure of nc rational functions with a given domain.

Recall that $\rx_0\subset \rx$ denotes the local subring of nc rational functions that are regular at the origin. As explained in Subsection \ref{ss:rat}, the domain of $\rr\in\rx_0$ is the complement of the free locus of a pencil corresponding to the minimal realization of $\rr$ by \cite[Theorem 3.1]{KVV1}. Hence Theorem \ref{t:main} yields the following result about comparable domains of elements in $\rx_0$.

\begin{cor}\label{c:rat}
	For $\rr,\rr'\in\rx_0$ let $(\cc,L_A,\bb)$ and $(\cc',L_{A'},\bb')$ be their minimal realizations. Then $\dom \rr\subseteq \dom \rr'$ if and only if there exists a homomorphism of $\kk$-algebras $\cA/\rad\cA \to \cA'/\rad\cA'$ induced by $A_i\mapsto A_i'$.
\end{cor}

\subsection{Regular nc rational functions}\label{ss:reg}

In this subsection we prove that every regular nc rational function, i.e., one that is defined at every matrix tuple, is in fact a polynomial. While this can be already deduced from Corollary \ref{c:rat}, we present a more precise proof which gives us explicit polynomial bounds for testing whether a nc rational function is a polynomial.

\begin{thm}\label{t:reg}
	Let $\rr$ be a nc rational function with minimal realization of size $d$ and let
	$$m=\left\{\begin{array}{ll}
	\lambda(d) & \kk \text{ is an algebraically closed field,}\\
	2\lambda(d) & \kk \text{ is a real closed field,}\\
	d\lambda(d)^2 & \text{otherwise.}
	\end{array}\right.$$
	If $\dom_m \rr=M_m(\kk)^g$, then $\rr$ is a nc polynomial of degree at most $d-1$.
\end{thm}

\begin{proof}
	Let $(\cc,L_A,\bb)$ be the minimal realization of $\rr$ about 0, i.e. $\rr=\cc L_A^{-1}\bb$. By \cite[Theorem 3.1]{KVV1}, $\det(L_A(X))\neq 0$ for every $X_i\in M_m(\kk)$. In particular,
	$$\det(I\otimes I-T\otimes \sum_iA_i\otimes Y_i)\neq 0$$
	for all $Y_i\in M_{\lambda(d)}(\kk)$ and $T\in M_k(\kk)$ with $k\le \frac{m}{\lambda(d)}$. Hence $\sum_iA_i\otimes Y_i$ is a nilpotent matrix by Remark \ref{r:const} and thus
	$$\det(I\otimes I-\sum_iA_i\otimes Y_i)=1$$
	for all $Y_i\in M_{\lambda(d)}(\kk)$. By Proposition \ref{p:nil}, the algebra generated by $A_1,\dots,A_g$ is nilpotent, so
	$$r=\cc\left(I-\sum_iA_ix_i\right)^{-1}\bb=\sum_{j=0}^{d-1}\cc\left(\sum_iA_ix_i\right)^j\bb$$
	is a polynomial.
\end{proof}

\subsubsection{Douglas' lemma for nc rational functions}

Douglas' lemma \cite[Theorem 1]{Dou} is a classical results in operator theory. Its finite-dimensional version states that for $A,B\in M_n(\C)$ we have $AA^*\le BB^*$ if and only if there exists $C\in M_n(\C)$ with $\|C\|\le1$, such that $A=BC$. As an application of the characterization of regular nc rational functions we give a version of Douglas' lemma for nc rational functions.

\begin{cor}\label{c:douglas}
	Let $\rr,\rs \in \rxc$. Then
	\begin{equation}\label{e:douglas}
	\rr(X)^*\rr(X)\le \rs(X)^*\rs(X)\qquad \text{for all}\ \ X\in \dom\rr\cap\dom\rs
	\end{equation}
	if and only if there exists $\lambda\in \C$, $|\lambda|\le 1$, such that $\rr=\lambda \rs$.
\end{cor}

\begin{proof}
	If $\rs=0$, then $\rr=0$, so we can assume that $\rs\neq0$. Denote
	$$\cD=\dom \rr\cap\dom\rs\cap\dom \rs^{-1}.$$
	By \eqref{e:douglas},
	$$(\rr(X)\rs^{-1}(X))^*(\rr(X)\rs^{-1}(X)) \le I\qquad \forall X\in\cD.$$
	Let $\rf=\rr\rs^{-1}$; then $\dom\rf\supseteq\cD$ and $\|\rf(X)\|\le 1$ for all $X\in\cD$. By definition, $\cD\cap M_n(\C)^g$ is Zariski open in $M_n(\C)^g$ and nonempty for infinitely many $n\in\N$, so boundedness implies $\dom_n\rf=M_n(\C)^g$ for infinitely many $n\in\N$. Consequently $\rf$ is regular everywhere, so it is a polynomial by Theorem \ref{t:reg}. Since it is bounded in norm by 1, it is constant by Liouville's theorem, so $\rr\rs^{-1}=\lambda\in\C$ and $|\lambda|\le 1$.
\end{proof}

\subsection{Characterization of nc rational functions with a given domain}\label{ss:char}

Let $\Dom=\{\dom\rr\colon \rr\in\rx_0\}$. A set in $\Dom$ is {\bf co-irreducible} if it is not an intersection of two larger sets in $\Dom$. Thus a domain is co-irreducible if and only if it is the complement of an irreducible free locus. A nc rational function $\rr\in\rx_0$ is {\bf irreducible} if it admits a realization $(\cc,L,\bb)$ with $L$ irreducible. Note that such a realization is automatically minimal by Remark \ref{r:simple}.

\begin{prop}\label{p:irr}
	If $\rr$ is irreducible, then $\dom\rr$ is co-irreducible. Conversely, for every co-irreducible set $D\in\Dom$ there exists a unique $d\in\N$ and a pencil $L$ of size $d$ such that irreducible rational functions whose domains are $D$ are exactly of the form
	$$\cc^t L^{-1}\bb,\qquad \bb,\cc\in\kk^d\setminus\{0\}.$$
\end{prop}

\begin{proof}
The first part follows from Proposition \ref{p:components}. Now let $\rr\in\rx_0$ and suppose that $D=\dom \rr'$ is co-irreducible. If $(\cc',L_{A'},\bb')$ is a minimal realization of $\rr'$, then $\cA'/\rad\cA'$ is simple by Proposition \ref{p:components}. Fix some irreducible representation $\rho:\cA'/\rad\cA'\to M_d(\kk)$. Let $A_i$ be the image of $A'_i$ under the homomorphism $\cA'\to M_d(\kk)$ and set $L=L_A$. Then $D$ is the complement of $\fl(L)$ by Corollary \ref{c:equal} and $D=\dom(\cc^t L^{-1}\bb)$ for every $\bb,\cc\neq0$. On the other hand, if $\rr''$ is an irreducible function with $\dom\rr=D$ and $(\cc'',L_{A''},\bb'')$ is its minimal realization, then $\fl(L)=\fl(L_{A''})$ and so $A''_i=PA_iP^{-1}$ for some $P\in\GL_d(\kk)$ by Theorem \ref{t:unique}. Hence $\rr=(P^t\cc'')L^{-1}(P^{-1}\bb'')$.
\end{proof}

Let $\cR(D)$ be the set of irreducible functions whose domains equal $D$. If we adopt the notation of Proposition \ref{p:irr}, then the elements of $\cR(D)$ are exactly nonzero linear combinations of $d^2$ linearly independent irreducible functions $\ee_\imath^t L^{-1}\ee_\jmath$ for $1\le\imath,\jmath\le d$.

The next lemma is essentially a version of Wedderburn principal theorem \cite[Theorem 2.5.37]{Row} for (possibly non-unital) $\kk$-subalgebras in $M_d(\kk)$.

\begin{lem}\label{l:WPT}
	Let $\cA\subseteq M_d(\kk)$ be a $\kk$-algebra, $\cA/\rad\cA\cong \cA^{(1)}\times\cdots\times \cA^{(s)}$ with $\cA^{(j)}$ simple, and let $\rho_j:\cA^{(j)}\to M_{d_j}(\kk)$ be irreducible representations. Then there exist a subalgebra $\cS\subseteq \cA$ and $P\in\GL_d(\kk)$ such that $\cA=\cS\oplus \rad \cA$ (as vector spaces) and $P\cS P^{-1}$ is precisely the image of
	$$\cA^{(1)}\times\cdots\times \cA^{(s)} \xrightarrow{(I\otimes\rho_1)\times\cdots\times (I\otimes\rho_s)}
	0\oplus (I\otimes M_{d_1}(\kk))\oplus\cdots \oplus (I\otimes M_{d_s}(\kk))\subseteq M_d(\kk).$$
\end{lem}

\begin{proof}
	If $\cA$ is unital, then Wedderburn's principal theorem yields the decomposition $\cA=\cS\oplus \rad\cA$, where $\cS\subseteq\cA$ is a subalgebra. If $\cA$ is not unital, let $\cA^\sharp$ be the unitization of $\cA$ \cite[Section 2.3]{Bre}; i.e., $\cA^\sharp=\kk\oplus\cA$, $\cA$ is an ideal of $\cA^\sharp$ and $\rad\cA^\sharp=\rad\cA$. Hence $\cA^\sharp=\cS'\oplus\rad\cA^\sharp$ by Wedderburn's principal theorem and therefore $\cA=\cA\cap(\cS'\oplus\rad\cA)=(\cA\cap \cS')\oplus\rad\cA$, so $\cS=\cA\cap\cS'$ is the required subalgebra.
	
	Since $\cS$ is semisimple, it has the multiplicative identity $\one_{\cS}$. Let $U=\ker \one_{\cS}$ and $V=\im\one_{\cS}$. Then $\kk^d=U\oplus V$, $\cS U=0$ and $\cS V\subseteq V$. Therefore we have a unital embedding
	$$\cA^{(1)}\times\cdots\times \cA^{(s)}\cong \cS\subseteq \End_{\kk}(V),$$
	so Lemma \ref{l:SN} applies.
\end{proof}

\begin{thm}\label{t:rat}
	Let $\rr\in\rx_0$. Then $\dom \rr=D_1\cap \cdots\cap D_s$ for some co-irreducible $D_j\in\Dom$ and $\rr$ is a nc polynomial in $\ulx\cup \cR(D_1)\cup\cdots\cup \cR(D_s)$ of degree at most $d$, where $d$ is the size of the minimal realization of $\rr$.
\end{thm}

\begin{proof}
Let $(\cc,L_A,\bb)$ be a minimal realization of $\rr$. Then $\dom \rr$ is a finite intersection of co-irreducible domains by Proposition \ref{p:components}. Let $\cA=\cS\oplus\rad\cA$ and $P\in\GL_d(\kk)$ be as in Lemma \ref{l:WPT}.  Write $A_i=S_i+N_i$ with respect to this decomposition and set $S=\sum_iS_ix_i$ and $N=\sum_iN_ix_i$. As a matrix over the ring of noncommutative formal power series $\kk\!\Langle\!\Langle X\Rangle\!\Rangle$, $L_A=I-S-N$ is invertible and
$$L_A^{-1}=(I-S)^{-1}\left(I-N(I-S)^{-1}\right)^{-1}=(I-S)^{-1}\sum_{j=0}^\infty \left(N(I-S)^{-1}\right)^j.$$
Since $(\rad \cA)^d=0$ and consequently
$$\left(N(I-S)^{-1}\right)^d=\left(\sum_{j=0}^\infty NS^j \right)^d=0,$$
we have
$$L_A^{-1}=(I-S)^{-1}\sum_{j=0}^{d-1} \left(N(I-S)^{-1}\right)^j.$$
Therefore $\rr$ is a polynomial of degree $d$ in $\ulx$ and the entries of $(I-S)^{-1}$. Let
$$PS_iP^{-1}=0\oplus(I\otimes A^{(1)}_i)\oplus\cdots\oplus (I\otimes A^{(s)}_i).$$
By the construction, $\cA^{(j)}$ is a simple algebra and $L_{A^{(j)}}$ is a simple pencil. Since
$$(I-S)^{-1}=P^{-1}\Big(
I\oplus (I\otimes L_{A^{(1)}}^{-1})\oplus\cdots \oplus (I\otimes L_{A^{(s)}}^{-1})
\Big)P,$$
the entries of $(1-S)^{-1}$ are polynomials of degree at most 1 in the elements of $\cR(D_1)\cup\cdots\cup \cR(D_s)$ by Proposition \ref{p:irr}.
\end{proof}

\begin{exa}
Let $\{x,y\}$ be our alphabet and consider rational functions
\begin{align*}
\rr_1&=(1-x-y(1-x)^{-1}y)^{-1}(1+x(1-x+y)^{-1})\\
&=\begin{pmatrix}0& 1& 0\end{pmatrix} \begin{pmatrix}1-x&-y& 0\\-y& 1-x& -x\\0& 0& 1-x+y\end{pmatrix}^{-1} \begin{pmatrix}0\\1\\1\end{pmatrix},\\
\rr_2&=(1-x-y)^{-1}(1-x)(1-x-y)^{-1}+(1-x-y)^{-1}x(1-x+y)^{-1}\\
&=\begin{pmatrix}1& 1& 0\end{pmatrix} \begin{pmatrix}1-x-y&-y& -x\\0& 1-x-y& 0\\0& 0& 1-x+y\end{pmatrix}^{-1} \begin{pmatrix}0\\1\\1\end{pmatrix}.
\end{align*}
It is easy to check that the given realizations are minimal, so
$$\dom \rr_1=\dom \rr_2=\dom \rs_1\cap \dom \rs_2,$$
where $\rs_1=(1-x-y)^{-1}$ and $\rs_2=(1-x+y)^{-1}$ are irreducible functions. It is evident that $\rr_2=\rs_1((1-x)\rs_1+x\rs_2)$ is a polynomial in $x,\rs_1,\rs_2$. On the other hand, it becomes clear that $\rr_1$ is a polynomial in $x,\rs_1,\rs_2$ only after we rewrite it as
$$\rr_1=\frac12((1-x-y)^{-1}+(1-x+y)^{-1})(1+x(1-x+y)^{-1})=\frac12(\rs_1+\rs_2)(1+x\rs_2).$$
\end{exa}


\def\flsn{\mathscr{Z}^{\operatorname{s}}_n}
\def\flhn{\mathscr{Z}^{\operatorname{h}}_n}

\section{Symmetric and hermitian pencils}\label{sec5}

In the final section we turn our attention to pencils with symmetric and hermitian matrix coefficients. Here the free loci are defined with tuples of symmetric and hermitian matrices, respectively. We call them free real loci. We investigate when two real loci are comparable; we show that this is equivalent to the existence of a $*$-homomorphism between $*$-algebras generated by the pencils (Theorem \ref{t:mainsh}). The main new ingredients needed to make this work are the theory of hyperbolic polynomials \cite{Gar,Ren} and the real Nullstellensatz from real algebraic geometry \cite{BCR}. Finally, in Subsection \ref{ss:kip} we formulate and prove a free (quantum) version of Kippenhahn's conjecture \cite{Kip} on simple eigenvalues of hermitian pencils.

Let $H_n(\C)\subseteq M_n(\C)$ and $S_n(\R)\subseteq M_n(\R)$ be the $\R$-spaces of hermitian and symmetric matrices, respectively. If the coefficients of $L$ are symmetric matrices, then $L$ is a {\bf symmetric} pencil and
$$\fls(L)=\bigcup_{n\in\N}\flsn(L),\qquad \flsn(L)=\fl_n(L)\cap S_n(\R)^g$$
is its {\bf free real locus}. Similarly, if the coefficients of $L$ are hermitian matrices, then $L$ is a {\bf hermitian} pencil with free real locus
$$\flh(L)=\bigcup_{n\in\N}\flhn(L),\qquad \flhn(L)=\fl_n(L)\cap H_n(\C)^g.$$

\subsection{Singularit\"atstellens\"atze for real loci}

In this subsection we prove the $*$-analog of Theorem \ref{t:main}.

\subsubsection{RZ polynomials}

Let $t$ and $\ulu=\{u_1,\dots,u_g\}$ be commutative indeterminates. Then $p\in\R[\ulu]$ is a {\bf real zero (RZ) polynomial} \cite{HV} if $p(0)\neq0$ and for every $\alpha\in\R^g$, $p(t\alpha)\in\R[t]$ has only real roots. This is essentially the dehomogenized version of hyperbolic polynomials that arise in convex optimization \cite{BGLS,Ren}, partial differential equations \cite{BGS} and real algebraic geometry \cite{Bra,NT,KPV}.

\begin{prop}\label{p:RZnull}
	Let $p\in\R[\ulu]$ be a RZ polynomial. If $q\in\C[\ulu]$ and $p(\alpha)=0$ implies $q(\alpha)=0$ for all $\alpha \in\R^g$, then $p(\alpha)=0$ implies $q(\alpha)=0$ for all $\alpha \in\C^g$.
\end{prop}

\begin{proof}
It clearly suffices to prove the statement for $q\in\R[\ulu]$. Let $p=p_1\cdots p_s$, where $p_j\in\R[\ulu]$ are irreducible. Fix $1\le j\le s$; then $p_j$ is a RZ polynomial. Since $p_j$ is also square-free, there obviously exist $\alpha,\beta\in\R^g$ such that $p_j(\alpha)p_j(\beta)<0$. By \cite[Theorem 4.5.1]{BCR}, the ideal in $\R[\ulu]$ generated by $p_j$ is real. Since $p_j(\alpha)=0$ implies $q(\alpha)=0$ for all $\alpha\in\R^g$, there exists $h_j\in \R[\ulu]$ such that $q=h_jp_j$ by the Real Nullstellensatz \cite[Theorem 4.1.4]{BCR}. Hence $q^s=(h_1\cdots h_s)p$, so $p(\alpha)=0$ implies $q(\alpha)=0$ for all $\alpha \in\C^g$.
\end{proof}

\subsubsection{Inclusion of free real loci}

Each symmetric or hermitian pencil $L$ gives rise to the RZ polynomial $\det L$. We now use the properties of RZ polynomials presented above to show that $\fls(L_1)\subseteq \fls(L_2)$ (or $\flh(L_1)\subseteq \flh(L_2)$) if and only if $\fl(L_1)\subseteq \fl(L_2)$.

\begin{prop}\label{p:dense}
	Let $L$ be a monic pencil.
	\begin{enumerate}[label={\rm(\roman*)}]
	\item If $L$ is hermitian, then $\flhn(L)$ is Zariski dense in $\fl_n(L)$ for every $n\in\N$.
	\item If $L_1$ and $L_2$ are symmetric, then
	$$\fls(L_1)\subseteq \fls(L_2)\quad \Rightarrow \quad\fl(L_1)\subseteq \fl(L_2).$$
	\end{enumerate}
\end{prop}

\begin{rem}
Note that $\flhn(L)$ is not Zariski dense in $\fl_n(L)$ if $L$ is symmetric and $n\ge2$. For example, if $\gX=(\gx_{ij})_{i,j=1}^2$ is a $2\times 2$ generic matrix, then the polynomial $(1-\gx_{11})(1-\gx_{22})-\gx_{12}^2$ vanishes on $\fls_2(1-x)$ but not on $\fl_2(1-x)$.
\end{rem}

\begin{proof}[Proof of Proposition \ref{p:dense}]
(i) Fix $n\in\N$ and an element of the coordinate ring of $M_n(\C)^g$, i.e., a complex polynomial $q$ in $gn^2$ variables. Assume that $q=0$ on $\flhn(L)$. For every $X_i,Y_i\in H_n(\C)$ let
$$p_{X,Y}:= \det(L(\ulu X+\ulv Y))\in \R[\ulu,\ulv],\qquad q_{X,Y}:=q(\ulu X+\ulv Y)\in \C[\ulu,\ulv].$$
By assumption we have
$$p_{X,Y}(\alpha,\beta)=0\ \Rightarrow\ q_{X,Y}(\alpha,\beta)=0 \qquad \forall \alpha,\beta\in\R^g.$$
Since $p_{X,Y}$ is a RZ polynomial, Proposition \ref{p:RZnull} implies
$$p_{X,Y}(\alpha,\beta)=0\ \Rightarrow\ q_{X,Y}(\alpha,\beta)=0 \qquad \forall \alpha,\beta\in\C^g.$$
If $Z\in M_n(\C)^g$ is arbitrary, then $Z=\frac12(Z+Z^*)+\frac12i(iZ^*-iZ)$ and $Z+Z^*,iZ^*-iZ$ are tuples of hermitian matrices, so $q=0$ on $\fl_n(L)$.

(ii) Let $\iota:\C\to M_2(\R)$ be the standard $*$-embedding of $\R$-algebras. For every $n\in\N$, the ampliation map
$$\iota_n=\id_{M_n(\R)}\otimes_{\R}\iota:M_n(\C)=M_n(\R)\otimes_{\R}\C\to M_{2n}(\R)$$
is again a $*$-embedding. If $L_1$ is symmetric and $X\in H_n(\C)^g$, then $L_1(X)$ is invertible if and only if $\iota_{dn}\left(L_1(X)\right)=L_1(\iota_n(X))$ is invertible. Therefore $\fls(L_1)\subseteq \fls(L_2)$ implies $\flh(L_1)\subseteq \flh(L_2)$ and the conclusion follows from considering $L_1$ and $L_2$ as hermitian pencils and applying (i).
\end{proof}

Let $L_A$ be a symmetric (resp. hermitian) pencil. As before, let $\cA$ denote the real (resp. complex) algebra generated by $A_1,\dots,A_g$. We claim that $\cA$ is semisimple. Indeed, suppose that $f(A)\in\rad\cA$ for some $f\in\pxr$ (resp. $f\in \pxc$). Since $f(A)^*\in\cA$, we have $f(A)^*f(A)\in\rad\cA$. In particular, $f(A)^*f(A)$ is a positive semi-definite nilpotent matrix, so $f(A)^*f(A)=0$ and thus $f(A)=0$.

\begin{thm}\label{t:mainsh}\hspace{0em}
	\begin{enumerate}[label={\rm(\roman*)}]
	\item Let $L_A$ and $L_B$ be symmetric pencils. Then $\fls(L_A)\subseteq\fls(L_B)$ if and only if there exists a $*$-homomorphism of $\R$-algebras $\cB\to\cA$ induced by $B_i\to A_i$.
	\item Let $L_A$ and $L_B$ be hermitian pencils. Then $\flh(L_A)\subseteq\flh(L_B)$ if and only if there exists a $*$-homomorphism of $\C$-algebras $\cB\to\cA$ induced by $B_i\to A_i$.
	\end{enumerate}
\end{thm}

\begin{proof}
Since $\cA$ and $\cB$ are semisimple, this assertion is a direct consequence of Proposition \ref{p:dense} and Theorem \ref{t:main}.
\end{proof}

Let $\Ortho_d\subset \GL_d(\R)$ and $\Unit_d\subset\GL_d(\C)$ be the orthogonal and the unitary group, respectively.

\begin{cor}\label{c:uniqueh}\hspace{0em}
	\begin{enumerate}[label={\rm(\roman*)}]
	\item Let $L_A$ and $L_B$ be symmetric minimal pencils of size $d$. Then $\fls(L_A)=\fls(L_B)$ if and only if there exists $Q\in\Ortho_d$ such that $B_i=QA_iQ^t$ for $i=1,\dots,g$.
	\item Let $L_A$ and $L_B$ be hermitian minimal pencils of size $d$. Then $\flh(L_A)=\flh(L_B)$ if and only if there exists $U\in\Unit_d$ such that $B_i=UA_iU^*$ for $i=1,\dots,g$.
	\end{enumerate}
\end{cor}

\begin{proof}
We prove just (i) since the proof of (ii) is analogous. If $\fls(L_A)=\fls(L_B)$, then by Theorem \ref{t:mainsh}(i) there exists a $*$-isomorphism $\cA\to\cB$ given by $A_i\mapsto B_i$. The rest follows as in the proof of Theorem \ref{t:unique} from the $*$-version of Lemma \ref{l:SN}, which in turn is a consequence of the following claim: if $\cC$ is a simple $\R$-algebra and $\iota,\iota':\cC\to M_d(\R)$ are irreducible $*$-representations, then there exists $Q\in\Ortho_d$ such that
\begin{equation}\label{e:SNstar}
\iota'(c)=Q\iota(c)Q^{-1}\qquad \forall c\in\cC.
\end{equation}
Indeed, by the Skolem-Noether theorem there exists $Q_0\in\GL_d(\R)$ such that \eqref{e:SNstar} holds. Because $\iota$ and $\iota'$ are $*$-homomorphisms, 
$$Q_0\iota(c)^tQ_0^{-1}=\iota'(c)^t=(Q_0\iota(c)Q_0^{-1})^t=Q_0^{-t}\iota(c)^tQ_0^t$$
holds for every $c\in\cC$. Therefore $Q_0^tQ_0$ lies in the centralizer of $\iota(\cC)$ in $M_d(\R)$. Since $\iota$ is irreducible representation, $Q_0^tQ_0$ belongs to the center of $M_d(\R)$, so $Q_0^tQ_0=\alpha I$ for $\alpha>0$ because $Q_0^tQ_0$ is positive-semidefinite. Now $Q=\frac{1}{\sqrt{\alpha}}Q_0\in \Ortho_d$ satisfies \eqref{e:SNstar}.
\end{proof}

In free real algebraic geometry an analogous result for free spectrahedra (distinguished convex sets associated to symmetric linear pencils) has been established in \cite{HKM} using nontrivial operator algebra techniques, e.g. Arveson's noncommutative Choquet boundary \cite{Arv}.

\subsection{Kippenhahn's free conjecture}\label{ss:kip}

Kippenhahn's conjecture \cite[Section 8]{Kip} can be restated as follows: if $H_1,H_2\in M_d(\C)$ are hermitian matrices that generate $M_d(\C)$ as a $\C$-algebra, then there exist $\alpha_1,\alpha_2\in\R$ such that the dimension of the kernel of $I-\alpha_1 H_1-\alpha_2 H_2$ is exactly one. While this conjecture has been established for matrices of small size \cite{Sha,Buc}, it is false in general by \cite{Laf}. However, we prove it is true in a free setting.

\begin{cor}\label{c:kipf}
	If $A_1,\dots,A_g\in M_d(\kk)$ generate $M_d(\kk)$ as $\kk$-algebra, then there exist $n\in\N$ and $X_1,\dots,X_g\in M_n(\kk)$ such that $\dim\ker L_A(X)=1$.
\end{cor}

\begin{proof}
By assumption there exists $f\in\px_+$ such that $f(A)=E_{1,1}$. By Lemma \ref{l:poly} there exist $X_i\in M_n(\kk)$ such that
\[1=\dim\ker(I-E_{1,1})=\dim\ker\left(I-\sum_i 0\cdot A_i-1\cdot f(A)\right)=\dim\ker L_A(X).\qedhere\]
\end{proof}

\subsubsection{Hermitian case} The original Kippenhahn's conjecture deals with hermitian matrices and their real linear combinations. Likewise, the free version can be strengthened for hermitian pencils.

\begin{cor}\label{c:kipq}
	If $A_1,\dots,A_g\in H_d(\C)$ generate $M_d(\C)$ as $\C$-algebra, then there exist $n\in\N$ and $X_1,\dots,X_g\in H_n(\C)$ such that $\dim\ker L_A(X)=1$.
\end{cor}

\begin{proof}
The set
$$\cO_n=\{X\in \fl_n(L_A)\colon \dim\ker L_A(X)=1 \}$$
is Zariski open in $\fl_n(L_A)$ and nonempty for some $n\in\N$ by Corollary \ref{c:kipf}. Since $\flhn(L_A)$ is Zariski dense in $\fl_n(L_A)$ by Proposition \ref{p:dense}, we have
\[\cO_n\cap \flhn(L_A)\neq\emptyset.\qedhere\]
\end{proof}

Similar reasoning as in Remark \ref{r:bounds} implies that $n\in\N$ from the statement of Corollary \ref{c:kipq} can be bounded by an exponential function in $g$ and $d$.

\subsubsection{Symmetric case}

Let $L_A$ be a symmetric pencil. In contrast to the hermitian case in Proposition \ref{p:dense}(i), $\flsn(L_A)$ is not Zariski dense in $\fl_n(L_A)$ for $n\ge2$. Hence we cannot use the same arguments as in Corollary \ref{c:kipq} to prove the real version of Kippenhahn's free conjecture. Nevertheless, we can at least deduce the following.

\begin{cor}\label{c:kipqs}
	If $A_1,\dots,A_g\in S_d(\R)$ generate $M_d(\R)$ as $\R$-algebra, then there exist $n\in\N$ and $X_1,\dots,X_g\in S_n(\R)$ such that $\dim\ker L_A(X)=2$.
\end{cor}

\begin{proof}
Since $A_1,\dots,A_g$ generate $M_d(\R)$ as $\R$-algebra, they also generate $M_d(\C)$ as $\C$-algebra. Hence there exist $X_1,\dots,X_g\in H_n(\C)$ such that $\dim\ker L_A(X)=1$ by Corollary \ref{c:kipq}. If $\iota_n:M_n(\C)\to M_{2n}(\R)$ is the $*$-embedding of $\R$-algebras from the proof of Proposition \ref{p:dense}(ii), then $\iota(X_i)\in S_{2n}(\R)$ and $\dim\ker L_A(\iota(X))=2$.
\end{proof}


\subsection*{Acknowledgments} The authors thank Bill Helton and Scott McCullough for valuable comments and suggestions.

\begin{appendices}
	
\section[Determinants of monic pencils and invariant theory]{Determinants of monic pencils and invariant theory,
by Claudio Procesi and \v Spela \v Spenko}\label{appA}

In this appendix we give a remark on comparing the determinants of monic pencils evaluated on generic matrices, that is closely related to comparison of free loci discussed in this paper. Let $\kk$ be an algebraically closed field of characteristic 0. First let us recall some basic facts of invariant theory; see e.g. \cite[Chapter 1]{MFK} or \cite[Chapter 14]{Pro1}.

Given a reductive group $G$ over $\kk$  and an affine $\kk$-variety $V$ on which $G$ acts
algebraically, one can construct the categorical quotient $V/\!/G$ with the quotient map $\pi:V\to V/\!/G$.
This is due to Hilbert in characteristic $0$, and Nagata-Mumford in general and it is done as follows. Let
$\cO(V)$ denote the coordinate ring of $V$. Then $G$ acts upon $\cO(V)$ and $\cO(V)^G$, the subalgebra
of $G$-invariants, is a finitely generated $\kk$-algebra (Hilbert's 14th problem). Then one defines $V/\!/G$ as
the variety with coordinate ring $\cO(V)^G$ and the map $\pi$ is defined by the inclusion $\cO(V)^G\subset \cO(V)$. The basic facts of the Hilbert-Mumford theory are the following.
\begin{enumerate}
\item The map $\pi$ is constant on $G$-orbits.
\item The map $\pi$ is surjective.
\item The variety $V/\!/G$ {\it parametrizes the closed $G$-orbits in $V$}, in the sense that for each $p\in V/\!/G$ there exists a unique closed orbit $O_p\subset V$ such that
$$\pi^{-1}(p)=\{v\in V\colon \overline{Gv}\supset O_p\}.$$
\end{enumerate}
In our case take $V=M_d(\kk)^g$ and let $G=\GL_d(\kk)$ act on $V$ by simultaneous conjugation. The coordinate ring of $V$ is the polynomial ring $\kk[\xi_{\imath\jmath}^{(j)}\colon 1\le j\le g, 1\le \imath,\jmath\le d]$. Now consider generic $d\times d$ matrices $\Xi_j=(\xi_{\imath\jmath}^{(j)})_{\imath\jmath}$ for $j=1,\dots,g$. In characteristic 0 the ring of invariants is generated by the elements $\tr w(\Xi)$ where $w$ runs over all words in 
the free monoid
$\mx$, by the first fundamental theorem of matrix invariants \cite[Theorem 11.8.1]{Pro1}. Moreover, 
the words 
 of length at most $d^2$ suffice \cite[final remark]{Raz} (or see \cite[Subsection 11.8.9]{Pro1}).\looseness=-1

Finally, one should think of the points in $M_d(\kk)^g$ as the set of $d$-dimensional representations of
the free algebra $\px$ and the equivalence under conjugation as isomorphism of representations. By a Theorem of M. Artin \cite[Section 12]{Art}, the closed orbits correspond to the semisimple representations, that is
the points $(A_1,\dots,A_g)$ whose coordinates generate a semisimple subalgebra of $M_d(\kk)$. In general
the closed orbit associated to a point $(A_1,\dots,A_g)$ is that of the semisimple representation associated
to the direct sum of the irreducible representations quotients of a Jordan-H\"older series.

Now consider two monic pencils of size $d$ associated to points $A = (A_1,\dots,A_g)$ and $B = (B_1,\dots,B_g)$. Suppose that $A$ and $B$ give the same point in the quotient $M_d(\kk)^g/\!/\GL_d(\kk)$, that is, their associated closed orbits coincide with the orbit of some semisimple point $\tilde{A} = (\tilde{A}_1,\dots,\tilde{A}_g)$. For arbitrary $m\in\N$ let $\Xi$ be a tuple of $m\times m$ generic matrices. Since $\det L_A(\Xi)$ is invariant under conjugation by $\GL_d(\kk)$ on its coefficients (and clearly polynomial in the entries of $A$) we have
$$\det L_A(\Xi) = \det L_{\tilde{A}}(\Xi) = \det L_B(\Xi).$$
The converse is given by the following theorem.

\begin{thm}\label{ta}
Let $A,B\in M_d(\kk)^g$. If $\Xi$ is a tuple of generic $d^2\times d^2$ matrices and $\det L_A(\Xi)=\det L_B(\Xi)$, then $A,B$ give rise to the same point in $M_d(\kk)^g/\!/\GL_d(\kk)$.
\end{thm}

\begin{proof}
If $t$ is an indeterminate, then by the hypothesis we have
$$\det\left(I-t\sum_j A_j\otimes \Xi_j\right)
=\det\left(I-t\sum_j B_j\otimes \Xi_j\right).
$$
This means that the two $d^3\times d^3$ matrices ${\bf A}=\sum_j A_j\otimes \Xi_j$ and ${\bf B}=\sum_j B_j\otimes \Xi_j$ have the same characteristic polynomial, hence for all $\ell$ we have
$$\tr({\bf A^\ell})=\tr({\bf B^\ell}).$$
Now
 $(\sum_{j=1}^g x_j)^\ell=\sum_{w\in\mx,|w|=\ell}w$ and $w(A\otimes \Xi)=w(A)\otimes w(\Xi)$. Hence for all $\ell$ we have
\begin{align*}
&\sum_{w\in\mx,|w|=\ell}\big(\tr w(A)-\tr w(B)\big) \tr w(\Xi) \\
=& \tr\left(\sum_{w\in\mx,|w|=\ell}\left(w(A)\otimes w(\Xi)-w(B)\otimes w(\Xi)\right)\right) \\
=& \tr({\bf A^\ell})-\tr({\bf B^\ell}) = 0.
\end{align*}
If $W_{\ell}$ denotes a set of representatives words of length $\ell$ with respect to the cyclic equivalence $w_1w_2\cyceq w_2w_1$, for instance Lyndon words, and $s_w$ denotes the number of words that are cyclically equivalent to $w$, then we have
$$\sum_{w\in W_{\ell}} s_w \big(\tr w(A)-\tr w(B)\big) \tr w(\Xi)
=\sum_{w\in\mx,\ |w|=\ell}\big(\tr w(A)-\tr w(B)\big) \tr w(\Xi)=0,$$
a pure trace identity on $M_{d^2}(\kk)$.
We know that there are no pure trace identities for $\ell\times \ell$ matrices of degree at most $\ell$ 
(see \cite[Proposition 8.3]{Pro} or the second fundamental theorem of matrix invariants \cite[Theorem 11.8.5]{Pro1}),
so we must have $\tr w(A) = \tr w(B)$ for all words $w$ of length at most $d^2$. But the elements $\tr w$
for all 
words $w$ of length at most $d^2$ generate the algebra of $\GL_d(\kk)$-invariants on $M_d(\kk)^g$, hence $\pi(A)=\pi(B)$ by the definition of $M_d(\kk)^g/\!/\GL_d(\kk)$.
\end{proof}

\begin{rem}\label{ra}
Since the free locus of a pencil is defined as the vanishing set of the determinant of the pencil, Theorem \ref{ta} is closely related to Theorem \ref{t:main}. However, Theorem \ref{ta} does not imply Theorem \ref{t:main} directly. The issue is that given a monic pencil $L_A$ and a tuple of generic $m\times m$ matrices $\Xi$, the polynomial $\det L_A(\Xi)$ might not generate the radical ideal associated with the hypersurface $\fl_m(L_A)$. It is easy to see that to derive Theorem \ref{t:main} from Theorem \ref{ta} it would suffice to know that for an irreducible pencil $L_A$ (that is, $A$ corresponds to an irreducible representation), $\det L_A(\Xi)$ is an irreducible polynomial for a tuple of large enough generic matrices. Fortunately this nontrivial statement holds by \cite[Theorem 3.4]{HKV}. Hence Theorem \ref{ta} together with the irreducibility result of \cite{HKV} gives an invariant-theory centric approach to the study of free loci.
\end{rem}

\end{appendices}


\end{document}